\documentclass[preprint,12pt,authoryear]{elsarticle}

\usepackage{amssymb}
\usepackage{graphicx}
\usepackage{enumerate}
\usepackage{amsthm}
\usepackage{amsfonts}
\usepackage{amsmath}
\usepackage{enumitem}
\usepackage{booktabs}
\usepackage{caption}
\usepackage{natbib}
\setcitestyle{numbers,square}
\usepackage{multirow}
\usepackage{tabularx}

\newcommand{\Tr}{{\rm Tr}}
\newcommand{\Ker}{{\rm Ker}}
\newcommand{\gf}{ {{\mathbb F}} }
\newtheorem{lemma}{Lemma}[section]
\newtheorem{proposition}{Proposition}[section]
\newtheorem{theorem}{Theorem}[section]
\newtheorem{remark}[theorem]{Remark}

\journal{Finite Field and Their Applications}

\begin{document}

\begin{frontmatter}



\title{The compositional inverses of permutation polynomials of the form $\sum_{i=1}^kb_i(x^{p^m}+x+\delta)^{s_i}-x$ over $\gf_{p^{2m}}$}


\author[wuaddress]{Danyao Wu\corref{mycorrespondingauthor}}
\cortext[mycorrespondingauthor]{Corresponding author}
\ead{wudanyao111@163.com}

\author[yuanaddress]{Pingzhi Yuan}
\ead{yuanpz@scnu.edu.cn}
\author[guanaddress]{Huanhuan Guan} 
\ead{guan1110h@163.com}

\author[liaddress]{Juan Li}
\ead{41170208@qq.com}

\address[wuaddress]{School of Computer Science and Technology,
	Dongguan University of Technology, Dongguan 523808, China}
\address[yuanaddress]{School of Mathematics, South China Normal University, Guangzhou 510631, China}
\address[guanaddress]{School of Mathematics and Statistics, Guizhou University of Finance and Economics, Guiyang 550025, China}
\address[liaddress]{School of Mathematics, Jiaying University, Meizhou, 514015, China}

%
%
%
%
%
%

%

\begin{abstract}
In this paper, we present the compositional inverses of several classes permutation polynomials of the form $\sum_{i=1}^kb_i(x^{p^m}+x+\delta)^{s_i}-x$ over $\gf_{p^{2m}}$, where for $1\leq i \leq k,$ $s_i, m$ are positive integers, $b_i, \delta \in \gf_{p^{2m}},$  and $p$ is prime.

\end{abstract}




\end{frontmatter}


\section{Introduction}
\label{}
Let  $\gf_q$ be the finite field with $q$ elements,  where $q$ is a prime power, and
let $\gf_q[x]$
be the ring of polynomials in a single indeterminate $x$ over $\gf_q$. A polynomial
$f \in\gf_q[x]$ is called a {\em permutation polynomial} of $\gf_q$ if its
associated polynomial mapping $f: c\mapsto f(c)$ from $\gf_q$ to itself is bijective. The unique polynomial denoted by $f^{-1}(x)$ over $\gf_q$
such that $f(f^{-1}(x))\equiv f^{-1}(f(x)) \equiv x \pmod{x^q-x}$ is called the compositional inverse of $f(x).$ Furthermore,  $f(x)$ is called  an involution when $f^{-1}(x)=f(x).$

The study of permutation polynomials and their compositional inverses over finite
fields in terms of their coefficients is a classical and difficult subject which
attracts people's interest partially due to their wide applications in coding theory
\cite{ding2013cyclic,ding2014binary,laigle2007permutation},
cryptography \cite{rivest1978method,schwenk1998public}, combinatorial design theory \cite{ding2006family}, and other areas of mathematics and engineering \cite{lidl1997finite,lidl1986introduction,mull1993permutation}. For instance, in block ciphers,
a permutation polynomial is usually used as an S-box to build the confusion layer and the compositional inverse of S-box comes into picture while decrypting the cipher. Both the permutation polynomial and its compositional inverse are implemented. Therefore, the explicit and efficient permutation polynomial and its compositional inverse are desired for designers. Indeed, gaining 
a better understanding of permutation polynomials and their compositional inverses in explicit format is not only meaningful, but also important for these applications.%
In general, it is difficult to discover new classes of permutation polynomials and computing the coefficients of
the compositional inverse of a permutation polynomial seems to be even more difficult, except for several classical classes such as
monomials, linearized polynomials, Dickson polynomials.
Compositional inverses of several classes of permutation polynomials
in explicit or implicit forms have been investigated in recent years,
which have nice structure.
We refer the readers to \cite{coulter2002compositional,li2019compositional,niu2021finding,tuxanidy2014inverses,tuxanidy2017compositional,wang2017note,baofengwu2013compositional,wu2014compositional,yuan2022compositional,yuan2022local,zheng2019inverses,zheng2019constructions}
for more details.

In 2014, Tuxanidy and Wang \cite{tuxanidy2014inverses}  studied the compositional inverse of a class of permutation polynomials of the form $f(x)=h(\psi(x))\varphi(x)+g(\psi(x))$ over $\gf_{q^n},$ where $\varphi, \psi, \bar{\psi}$ are additive polynomials satisfying $\varphi \circ \psi= \bar{\psi} \circ \varphi$ and $\sharp\psi(\gf_{q^n})=\sharp\bar{\psi}(\gf_{q^n}).$ Here,  $g, h \in \gf_{q^n}[x]$ are polynomials with $h(\psi(\gf_{q^n}))\subseteq \gf_{q}^*.$  The inverse of $f(x)$ is related to the inverses $\bar{f}(x)$ and $\varphi|_{\ker(\psi)}$, where $\bar{f}(x)=h(x)\varphi(x)+\bar{\psi(g(x))}$ is a bijection from $\psi(\gf_{q^n})$ to $\bar{\psi}(\gf_{q^n}).$ 

In 2021, Reis and Wang \cite{reis2021permutation} investigated the permutation polynomial of the form $P_1(x)=g(\Tr_{q^n/q}(x))+k(\Tr_{q^n/q}(x))L_h(x)$ and its inverse over $\gf_{q^n}$,  where  $g(x) \in \gf_{q^n}[x], h(x)\in \gf_q[x],$ $k(x) \in \gf_q[x]$ with  $k(\gf_q)\subseteq \gf_q^*.$ Here,  $L_h(x)$ is the linearized $q$-associate of $h.$ The compositional inverse of $P_1(x)$ is related to the inverse of  another permutation polynomial $Q(x)$ over $\gf_{q}.$ Later, they \cite{reis2024constrcuting} refined their results by  studying the permutation polynomial of $P_2(x)$ and its inverse over $\gf_{q^n}.$ In this case, $P_2(x)$ is obtained from $P_1(x)$  by replacing $L_l(x)$ with $\Tr_{q^n/q}(x)$,  and $L_l(x)$ is the $q$-associated of $l(x)=(x^n-1)/(x^t-a)$ with $t|n$ and $ a^{n/t}=1.$ 

Many series of explicit  permutation polynomials of the form $\sum_{i=1}^kb_i(x^{p^m}+ax+\delta)^{s_i}-ax$ $(a^{q+1}=1)$ over $\gf_{q^2}$ \cite{li2023several,li2024several,li2018permutation,wu2022some,wu2023some,xu2022several,yuan2015permutation,zeng2017permutation}  are special cases of $f(x)$ in \cite[Theorem 1.2]{tuxanidy2014inverses} ( where $n=2,$ $h(x)=1$, $\varphi(x)=-ax,$ $\psi(x)=x^q+ax,$ $\bar{\psi}(x)=ax^q+x$, and $g(x)=\sum_{i=0}^k(x+\delta)^{s_i}$),  or  $P_1(x)$ in \cite[Theorem 3.2]{reis2021permutation} approximately (see \eqref{P(x)} and $n=2,$  $k(x)=1, $ $L_h(x)=-x$,  $g(x)=\sum_{i=0}^k(x+\delta)^{s_i}$). However, finding $
\bar{f}^{-1}(x)$ in \cite[Theorem 1.2]{tuxanidy2014inverses} and $Q^{-1}(x)$ in \cite[Theorem 3.2]{reis2021permutation} are still challenging. In this paper, we will explicitly determine  the compositional inverses of the permutation polynomials of the form $\sum_{i=1}^kb_i(x^{p^m}+x+\delta)^{s_i}-x$ over $\gf_{q^2}$ with different choices of the exponents $s_i.$

The remainder of this paper is organized as follows. Section 2  introduces  related results. Next, Section 3 presents the investigation of the compositional inverses of permutation polynomials in the form $\sum_{i=1}^kb_i(x^{p^m}+x+\delta)^{s_i}-x$ over $\gf_{p^{2m}}.$

\section{Auxiliary results }
In this section, we present some auxiliary results that will be needed in the
sequel.

\begin{lemma}\cite[Theorem 5.1 ]{akbary2011constructing}\label{th5.2} For any polynomial $g \in \gf_{q^n}[x], $ any  additive polynomials 
 $\varphi, \psi, \bar{\psi}$ satisfying $\varphi \circ \psi= \bar{\psi} \circ \varphi$ and $\sharp\psi(\gf_{q^n})=\sharp\bar{\psi}(\gf_{q^n}),$ and any polynomial $h \in \gf_{q^n}[x]$ such that  $h(\psi(\gf_{q^n}))\subseteq \gf_{q}^*.$ Then $$P(x)=h(\psi(x))\varphi(x)+g(\psi(x))$$ permutes $\gf_{q^n}$ if and only if 
 
 (i) $\ker(\varphi)\cap \ker(\psi)=\{0\},$ or, equivalently, $\varphi$ induces a bijection from $\ker(\psi)$ to $\ker(\bar{\psi});$   and 
 
 (ii) $\tau(x)=h(x)\varphi(x)+\bar{\psi}(g(x))$ is a bijection from $\psi(\gf_{q^n})$ to $\bar{\psi}(\gf_{q^n}).$
\end{lemma}
		
Even though $\varphi(x)\circ \psi(x)=\bar{\psi}(x) \circ \varphi(x)$ and $\sharp \psi(\gf_{q^n})=\sharp \bar{\psi}(\gf_{q^n})$ hold, we still may not  have a "nice enough" expression for   
$\Ker(\psi).$ 
To calculate the inverse of $f(x)$ in Lemma \ref{th5.2},   Tuxanidy and Wang \cite{tuxanidy2014inverses}  added the conditions $\sharp S_{\psi}=\sharp S_{\bar{\psi}}$ and $\ker(\varphi)\cap \psi(S_{\psi})=\{0\},$ 
where  $S_{\psi}=\{x-\psi(x)\mid x\in \gf_{q^n}\}$ and $S_{\bar{\psi}}=\{x-\bar{\psi}(x)\mid x\in \gf_{q^n}\}.$  Then they have the following result.

\begin{lemma}\cite[Theorem 1.2]{tuxanidy2014inverses}\label{th1.2}
	Using the same notation and assumptions of Lemma \ref{th5.2}, assume that $P(x)$ is a permutation of $\gf_{q^n}.$ Let $\tau^{-1}, $ $\varphi^{-1}|_{S_{\bar{\psi}}}$ induce the inverses of $\tau|_{\psi(\gf_{q^n})}$ and $\varphi|_{S_{\psi}},$ respectively. Then the compositional inverse of $P(x)$ on $\gf_{q^n}$ is given by 
	$$P^{-1}(x)=\tau^{-1}(\bar{\psi}(x))+\varphi^{-1}|_{S(\bar{\psi})}\left(\frac{x-\bar{\psi}(x)-g(\tau^{-1}(\bar{\psi}(x)))+\bar{\psi}(g(\tau^{-1}(\bar{\psi}(x))))}{h(\tau^{-1}(\bar{\psi}(x)))}\right).$$ 
	Furthermore, if $\varphi$ induces a bijection from $\psi(\gf_{q^n})$ to $\bar{\psi}(\gf_{q^n}),$ then $\varphi$ permutes $\gf_{q^n}$ and the compositional inverse of $P(x)$ on $\gf_{q^n}$ is given by 
		$$P^{-1}(x)=\varphi^{-1}\left(\frac{x-g(\tau^{-1}(\bar{\psi}(x)))}{h(\tau^{-1}(\bar{\psi}(x)))}\right).$$ 
 \end{lemma}

Now, taking $h(x)=1,$ $\varphi(x)=-x,$ $\psi(x)=\bar{\psi}(x)=\Tr_{q^2/q}(x),$ and $g(x)=\sum_{i=1}^kb_i(x+\delta)^{s_i}$ in Lemmas \ref{th5.2} and \ref{th1.2},  then we have the following lemma. 
\begin{lemma}\label{th1}
	Let $q$ be a prime power.
	For $i=1, 2, \cdots, d,  $ assume that $b_i, \delta \in \gf_{q^2}$ and $s_i$ are positive integers. 
	Then the polynomial  $$P(x)=\sum_{i=1}^{k}b_i(x^q+x+\delta)^{s_i}-x$$ permutes $\gf_{q^2}$ if and only if 
	\begin{equation}\label{letau}
		\tau(x)=\Tr_{q^2/q}(x) \circ g(x)-x
		\end{equation}
	permutes $\gf_q.$ Moreover, if $p^{-1}(x)$ exists over $\gf_{q^2},$ then 
	$$P^{-1}(x)=g(x)\circ \tau^{-1}(x)\circ \Tr_{q^2/q}(x)-x,$$ where $g(x)=\sum_{i=1}^kb_i(x+\delta)^{s_i}.$
	\end{lemma}
	
Lemma \ref{th1}	is crucial in this paper, which will be frequently used in  next section.

We list two results about the compositional inverses of linearized permutation polynomials at last. 
 
For a positive integer $m, $ 
define a sequence 
$$S_{-1}=0, S_0=1, S_i=b^{2^{i-1}}S_{i-1}+a^{2^{i-1}}S_{i-2},$$ 
where $1\leq i\leq m$ and $a, b \in \gf_{2^m}.$ Y. Zheng, Q. Wang and W. Wei \cite{zheng2019inverses} studied  the inverse of linearized polynomial of the form $x^4+bx^2+ax$ over $\gf_{2^m}.$
\begin{lemma}\label{421} \cite[Corollary 4]{zheng2019inverses}
Let $L(x)=x^4+bx^2+ax$, where $a ,b \in \gf_{2^m}$ and $m>1.$ Then $L(x)$ is a permutation polynomial over $\gf_{2^m}$ if and only if $S_m+aS_{m-2}^2=1.$ Moreover, if $L(x)$ permutes $\gf_{2^m},$ the inverse of $L(x)$ over $\gf_{2^m}$ is given by  
$$L^{-1}(x)=\sum_{i=0}^{m-1}(S_{m-2-i}^{2^{i+1}}+a^{1-2^{i+1}}S_i)x^{2^i}.$$
	
	\end{lemma}

\begin{lemma}\label{binomial}\cite[Theorem 2.1]{baofengwu2013compositional}
Let $L_r(x)=x^{q^r}-ax$, where $a\in \gf_{q^m}^*$, and $1\leq r\leq m-1$. Then $L_r(x)$ is a permutation polynomial over $\gf_{q^m}$ if and only if the norm $N_{q^m/q^d}(a)\neq 1,$ where $d=gcd(m ,r).$ In this case, its inverse on $\gf_{q^m}$ is 
$$L_r^{-1}(x)=\frac{N_{q^m/q^d}(a)}{1-N_{q^m/q^d}(a)}\sum_{i=0}^{m/d-1}a^{-\frac{q^{(i+1)r}-1}{q^r-1}}x^{q^{ir}}.$$
\end{lemma}

\section{Main results}

 We consider any $a\in \gf_{q^2}$ with $a^{q+1}=1.$ Let $\varepsilon \in \gf_{q^2}$ be  a primitive element of $\gf_{q^2}$ and $t$ be the positive integer with
$a=\varepsilon^{(q-1)t}$ and $0\leq t\leq q.$ 
Put $x=\varepsilon^ty,$ we have that 
\begin{align}\label{P(x)}
	\sum_{i=1}^{k}b_i(x^q+ax+\delta)^{s_i}-ax=&\,
	\sum_{i=1}^{k}b_i(x^q+\varepsilon^{(q-1)t}x+\delta)^{s_i}-\varepsilon^{(q-1)t}x\nonumber\\
	=&\,\sum_{i=1}^{k}b_i\left(\varepsilon^{qt}(y^q+y+\varepsilon^{-qt}\delta)\right)^{s_i}-\varepsilon^{qt}y\nonumber\\
	=&\, \varepsilon^{qt}\left(	\sum_{i=1}^{k}\bar{b}_i(y^q+y+\bar{\delta})^{s_i}-y\right),
\end{align}
Where $\bar{b}_i=b_i\varepsilon^{qt(s_i-1)}$ and $\bar{\delta}=\varepsilon^{-qt}\delta.$ Hence, it suffices to consider the compositional inverses of  the permutation polynomials of  the form 

$$P(x)=\sum_{i=1}^{k}b_i(x^q+x+\delta)^{s_i}-x$$ 
over $\gf_{q^2}.$

We  give a proposition at first. 
\begin{proposition}\label{th11}
	Let $q$ be a prime power.  For $1\leq i \leq k,$  assume that $b_i, \delta \in \gf_{q^2}$ and $s_i$ are positive integers.  Then the polynomial $ P_1(x)=\sum_{i=1}^{k}b_i(x^q+x+\delta)^{s_i}-x$ is a permutation polynomial over $\gf_{q^2}$ if and only if $P_2(x)=\sum_{i=1}^{k}b_i^q(x^q+x+\delta)^{qs_i}-x$ permutes $\gf_{q^2}$. Moreover, if $P_1(x)$ is a permutation polynomial over $\gf_{q^2}$, then   $$P_1^{-1}(x)+x=(P_2^{-1}(x)+x)^q,$$
	where $P_1^{-1}(x)$ and $P_2^{-1}(x)$ are the compositional inverses of $P_1(x)$ and $P_2(x)$, respectively.
\end{proposition}
\begin{proof}
	Since $P_2(x^q)=\sum_{i=1}^{k}b_i^q(x^q+x+\delta)^{qs_i}-x^q=P_1(x)^q,$ we draw the first conclusion  immediately.  
	
	Moreover, if $P_1(x)$ permutes $\gf_{q^{2}}$, then $P_2(x)$ permutes $\gf_{q^2}.$ It follows from  Lemma \ref{th1}  that 
	$$P_1^{-1}(x)=\left(\sum_{i=1}^{k}b_i(x+\delta)^{s_i}\right)\circ \tau^{-1}(x)\circ  \Tr_{q^2/q}(x)-x$$ and $$P_2^{-1}(x)=\left(\sum_{i=1}^{k}b_i^q(x+\delta)^{qs_i}\right)\circ  \tau^{-1}(x)\circ  \Tr_{q^2/q}(x)-x,$$ where $\tau(x)=\Tr_{q^2/q}(x) \circ \left(\sum_{i=1}^kb
	_i(x+\delta)^{s_i}\right)-x$ and $\tau^{-1}(x)$ is the compositional inverse of $\tau(x)$ over $\gf_q.$ 
	Hence, 
	$$P_1^{-1}(x)+x=(P_2^{-1}(x)+x)^q.$$
	We are done. 
\end{proof}
\begin{remark} By Proposition \ref{th11}, we will only select either $P_1(x)$ or  $P_2(x)$ for our study in the following subsections. 
	\end{remark}

According to  Lemma \ref{th1}, it can be observed  that obtaining the  compositional inverse of $\tau(x)$ over $\gf_q$ enables us to determine the compositional inverse of $P(x)$ over $\gf_{q^2}.$ Thus, we split the subsequent subsections based on the form of  $\tau(x)$, $\tau^2(x)$, or $\tau^4(x).$  In each subsection, we provide a detailed proof for one representative result and only present the equations of $\tau(x)$, $\tau^2(x)$, or $\tau^4(x)$
 for the other theorems.

\subsection{ The form of $\tau(x)$, $\tau^2(x)$ or $\tau^4(x)$ is $Ax^{2}+C$ or $Ax+B$
}\label{sec3.1}
	In following two results, we will give two classes of involution of the form $\sum_{i=1}^{d}b_i(x^q+x+\delta)^{s_i}-x$ over $\gf_{q^2}.$ The results can be obtained directly from Lemma \ref{th1}. We omit the details here. 
\begin{theorem}
	Let $q$ be a power of $2.$ For $i=1, \cdots, k, $ assume that $b_i\in \gf_q^*, \delta \in \gf_{q^2}$ and $s_i$ are positive integers with $s_iq\equiv s_i \pmod{q^2-1}$. Then the polynomial
	$$P(x)=\sum_{i=0}^kb_i(x^q+x+\delta)^{s_i}+x$$  is an involution over $\gf_{q^2}$.
\end{theorem}
\begin{theorem}
	Let $q$ be a power of $2.$ For $i=1, \cdots, k, $ assume that $b_i, \delta \in \gf_{q}^*$ and $s_i$ are positive integers. Then the polynomial
	$$P(x)=\sum_{i=0}^kb_i(x^q+x+\delta)^{s_i}+x$$  is an involution over $\gf_{q^2}$.
\end{theorem}
\begin{remark}
	Based on the above theorem, we only consider $\delta \notin \gf_q$ when  calculating  the compositional inverse of the permutation polynomials of the form  $P(x)=\sum_{i=0}^kb_i(x^q+x+\delta)^{s_i}+x$ over $\gf_{q^2}$ with  $q$ being a power of $2$ and $b_i \in \gf_{q}$ in subsections \ref{sec3.1} and \ref{sec3.2}.
\end{remark}

\begin{theorem}\label{th28}
	Let $q=2^m\geq4$ with a  positive integer $m$. Assume that  $ \delta \in \gf_{q^2}$  and $b\in \gf_q^*$ with $b^4\Tr_{q^2/q}(\delta)=1.$   Then the compositional inverse of  
	$$P(x)=b(x^q+x+\delta)^{q(2q+3)/4}+x$$ over  $\gf_{q^2}$ is $$P^{-1}(x)=x+b\left(\Tr_{q^2/q}(\delta)^{-1}(x^q+x)^2+\delta^{q+1}\Tr_{q^2/q}(\delta)^{-1}+\delta\right)^{q(2q+3)/4}.$$
\end{theorem}

\begin{proof}
	\cite{wu2023some} had shown that if $b^4\Tr_{q^2/q}(\delta)=1,$ then the polynomial $P(x)=b(x^q+x+\delta)^{q(2q+3)/4}+x$ is a permutation polynomial over  $ \gf_{q^2}.$ Therefore, by Lemma \ref{th1},  $\tau(x)=b(x+\delta)^{q^2(2q+3)/4}+b(x+\delta)^{q(2q+3)/4}+x$ permutes  $\gf_{q}. $
	Moreover,  since $b^4\Tr_{q^2/q}(\delta)=1,$  we have
	\begin{align*}
		\tau(x)^4=&\,b^4(x+\delta)^{2q+3}+b^4(x+\delta)^{3q+2}+x^4\\
		=&\, b^4\Tr_{q^2/q}(\delta)(x^2+\delta^2)(x^2+\delta^{2q})+x^4\\
		=&\, \Tr_{q^2/q}(\delta)^2x^2+\delta^{2q+2},			
	\end{align*}
	or \begin{equation*}
		\tau(x)=x^{q/4}\circ (\Tr_{q^2/q}(\delta)^2x+\delta^{2q+2})\circ x^2.
	\end{equation*}
	This yields 
	\begin{align*}
		\tau^{-1}(x)=&\, x^{q/2}\circ (\Tr_{q^2/q}(\delta)^{-2}x+\delta^{2q+2}\Tr_{q^2/q}(\delta)^{-2})\circ x^4\\
		=&\, \Tr_{q^2/q}(\delta)^{-1}x^2+\delta^{q+1}\Tr_{q^2/q}(\delta)^{-1}.
	\end{align*}
	Hence,  by Lemma \ref{th1}, the compositional inverse of $P(x)$ over $\gf_{q^2}$ is 
	$$P^{-1}(x)=x+b\left(\Tr_{q^2/q}(\delta)^{-1}(x^q+x)^2+\delta^{q+1}\Tr_{q^2/q}(\delta)^{-1}+\delta\right)^{q(2q+3)/4}.$$
	We are done. 
\end{proof}

Given $\tau(x)^2= \Tr_{q^2/q}(\delta)x+\delta^{q+1},$  we have the following result. 
\begin{theorem}
	Let $q$ be a power of $2.$ Assume that $b \in \gf_{q}^*$ and $ \delta\in \gf_{q^2}$ with $b\Tr_{q^2/q}(\delta)=1$. 	Then the compositional inverse of $$P(x)=b(x^q+x+\delta)^{\frac{q^2+q}{2}+1}+x$$  over $\gf_{q^2}$ is  
	\begin{align*}
		P^{-1}(x)=&\, x+b\left(\Tr_{q^2/q}(\delta)^{-1}(x^q+x)^2+\Tr_{q^2/q}(\delta)^{-1}\delta^{q+1}+\delta\right)^{1+(q^2+q)/2}.
	\end{align*}
\end{theorem}
Similarly, we have following theorem. We only give $\tau^2(x)=(\Tr_{q^2/q}(b)^2+1)x^2+\Tr_{q^2/q}(b)^2\Tr_{q^2/q}(\delta)x+\delta^{q+1}\Tr_{q^2/q}(b)^2$ here.

\begin{theorem}\label{tha4}	
	Let $q$ be a power of $2.$  Assume that $b\in \gf_{q^2}^*$ and $\delta \in \gf_{q^2}.$ Let the polynomial
	$$P(x)=b(x^q+x+\delta)^{q(q+1)/2}+x$$ permute $\gf_{q^2}.$ \\
	If $b\in \gf_q$ or $\delta \in \gf_q$, then the compositional inverse of $P(x)$ over $\gf_{q^2}$ is 
	$$P^{-1}(x)=b\left(\frac{x^q+x}{\Tr_{q^2/q}(b)+1}+\frac{\delta^{(q+1)q/2}\Tr_{q^2/q}(b)}{\Tr_{q^2/q}(b)+1}+\delta\right)^{q(q+1)/2}+x.$$	
	If $\Tr_{q^2/q}(b)=1, $ then the compositional inverse of $P(x)$ over $\gf_{q^2}$ is 
	$$P^{-1}(x)=b\left(\Tr_{q^2/q}(b)^{-1}(x^q+x)^2+\Tr_{q^2/q}(b)^{-1}+\delta\right)^{q(q+1)/2}+x.$$	
\end{theorem}
Given $\tau(x)=b\Tr_{q^2/q}(\delta)x^2+b\delta^{q+1}\Tr_{q^2/q}(\delta), $ we have the following theorem.

\begin{theorem}
	Let  $q$ be a power of $2.$ Assume that $b \in \gf_{q}^*$ and $ \delta \in \gf_{q^2}$ with $b\Tr_{q^2/q}(\delta)^2=1.$ Then 
	the compositional inverse of
	$$P(x)=b(x^q+x+\delta)^{q+2}+x$$ over $\gf_{q^2}$ is \begin{align*}
		P^{-1}(x)=&\,b\left(\left(\Tr_{q^2/q}(\delta)(x^q+x)+\delta^{q+1}\right)^{q/2}+\delta \right)^{q+2}+x.
	\end{align*}
\end{theorem}

Given  $\tau(x)= b\Tr_{q^2/q}(\delta)^{2^i}x^{2}+b\delta^{q+1}\Tr_{q^2/q}(\delta)^{2^i}$ here, we have following result.
\begin{theorem}\label{th24}
	Let $q=2^m$ with a  positive integer $m.$ Assume that  $ \delta \in  \gf_{q^2}$, $b\in \gf_{q}^*$ and $i$ is a non-negative  integer with $b\Tr_{q^2/q}(\delta)^{2^i+1}=1. $ 
	Then the compositional inverse of $P(x)=b(x^q+x+\delta)^{2^i+q+1}+x$ over $\gf_{q^2}$ is 
	$$P^{-1}(x)=x+b\left(\left(\Tr_{q^2/q}(\delta)(x^q+x)+\delta^{q+1}\right)^{q/2}+\delta\right)^{2^i+q+1}.$$
\end{theorem}

G. Li and X. Cao \cite{li2023several} presented that the polynomial $f(x)=(x^{p^m}-x+\delta)^{l(p^m+1)+p^j}+(x^{p^m}-x+\delta)^{l(p^m+1)+p^{m+j}}+x$ is a permutation polynomial over $\gf_{p^{2m}},$ where $p$ is prime, and $l, j $ and $m$ are positive integers in Proposition 8. We improve this result and give the compositional inverse of permutation polynomials of this form in the following theorem.  

\begin{theorem}\label{thq2}
	Let $q$ be an odd prime power. For $1\leq i \leq k, $  $b_i\in \gf_q,$  $\delta\in \gf_{q^2},$ $l_i,$  $s_i $ are positive integers with $2\nmid l_i$. Let $\varepsilon$ be a primitive element of $\gf_{q^2}$ and $t=(q+1)/2.$ Then
	the polynomial 
	$$P(x)=-x+\sum_{i=1}^kb_i\varepsilon^{tl_i}\left((x^q+x+\delta)^{s_i}+(x^q+x+\delta)^{qs_i}\right)$$ 
	permutes $\gf_{q^2}$ and  the compositional inverse of  $P(x)$ over $\gf_{q^2}$ is  
	$$P(x)^{-1}=-x+\sum_{i=1}^kb_i\varepsilon^{tl_i}\left((-x^q-x+\delta)^{s_i}+(-x^q-x+\delta)^{qs_i}\right).$$
\end{theorem}

Similarly, we have the following result. We only give $\tau(x)=\delta\sum_{i=0}^k\Tr_{q^2/q}(b_i)+\left(\sum_{i=0}^k\Tr_{q^2/q}(b_i)-1\right)x$ here. 
\begin{theorem}\label{thq3}
	Let $q$ be a prime power.  For $1\leq i\leq k$, assume that $\delta \in \gf_q,$  $b_i \in \gf_{q^2}$ and $s_i$ are positive integers. Then the polynomial
	$$P(x)=\sum_{i=1}^kb_i(x^q+x+\delta)^{s_i(q-1)+1}-x$$
	permutes $\gf_{q^2}$ if and only if $\sum_{i=0}^k\Tr_{q^2/q}(b_i)\neq1$
	Moreover, if $P(x)$ permutes $\gf_{q^2}$, the compositional inverse of $P(x)$ over $\gf_{q^2}$ is
	$$P^{-1}(x)=\sum_{i=0}^kb_i\left(\frac{x^q+x}{\sum_{i=0}^k\Tr_{q^2/q}(b_i)-1}+\frac{\delta\sum_{i=0}^k\Tr_{q^2/q}(b_i)}{\sum_{i=0}^k\Tr_{q^2/q}(b_i)-1}+\delta\right)^{s_i(q-1)+1}-x.$$
\end{theorem}

\subsection{ The form of $\tau(x)$ or $\tau^4$ is $Ax^{4}+Bx^2+Cx+D $ 
}\label{sec3.2}

Throughout this subsection, let $q=2^m$ with a positive integer $m,$ $b\in\gf_{q}^*,$ and $\delta \in \gf_{q^2}\setminus \gf_q.$
\begin{theorem}
	If $q\geq 4$ and the polynomial
	$$P(x)=b(x^q+x+\delta)^{1+(q^2+q)/4}+x$$ permutes $\gf_{q^2},$ then the compositional inverse of $P(x)$ over $\gf_{q^2}$ is 
	$$P^{-1}(x)=x+b\left(\delta+\sum_{i=0}^{m-1}\left(S_{m-2-i}^{2^{i+1}}+D^{1-2^{i+1}}S_i\right)((x^q+x)^4+B)^{2^i}\right)^{1+(q^2+q)/4},$$	 	where
	$B=b^4(\delta^q+\delta)^4\delta^{q+1}, C=b^4(\delta^q+\delta)^4, D=b^4(\delta^q+\delta)^5$ and $S_i$ is a sequence with $S_{-1}=0,$ $S_0=1, $ $S_i=C^{2^{i-1}}S_{i-1}+D^{2^{i-1}}S_{i-2}.$
	
\end{theorem}

\begin{proof}
	
	Since the polynomial $P(x)=b(x^q+x+\delta)^{1+(q^2+q)/4}+x$ permutes $\gf_{q^2}$, we obtain  that
	$\tau(x)=b(x+\delta)^{q+(q^3+q^2)/4}+b(x+\delta)^{1+(q^2+q)/4}
	+x$  permutes $\gf_{q}$ by Lemma \ref{th1},  and the latter condition is equivalent to $\tau(x)^4$ is a permutation polynomial over  $\gf_{q}.$ 		Furthermore, we have 
	\begin{align}\label{1eq1+(q^2+q)/4}
		\tau(x)^4=&\,b^4(x+\delta)^{5q+1}+b^4(x+\delta)^{q+5}
		+x^4\nonumber\\
		=&\, b^4(x+\delta)^{q+1}\left((x+\delta)^{4q}+(x+\delta)^{4}\right)+x^4\nonumber\\
		=&\,x^4+b^4(\delta^q+\delta)^4x^2+b^4(\delta^q+\delta)^5x
		+b^4(\delta^q+\delta)^4\delta^{q+1}
	\end{align}	
	Put $B=b^4(\delta^q+\delta)^4\delta^{q+1}, C=b^4(\delta^q+\delta)^4, D=b^4(\delta^q+\delta)^5.$
	Then by \eqref{1eq1+(q^2+q)/4}, we have
	\begin{align*}
		\tau(x)^4
		=&\, (x+ B)\circ	(x^4+Cx^2+Dx),
	\end{align*}
	or \begin{equation} \label{1eq1+(q
			^2+q)/4}
		\tau(x)=x^{q/4}\circ(x+ B)\circ	(x^4+Cx^2+Dx).
	\end{equation}
	Since  $\tau(x)$ permutes  $\gf_{q},$ it follows from Lemma \ref{421} and \eqref{1eq1+(q
		^2+q)/4} that the compositional inverse of $\tau(x)$ over $\gf_q$ is 
	\begin{align}\label{2eq1+(q
			^2+q)/4}
		\tau^{-1}(x)=&\,(x^4+Cx^2+Dx)^{-1}\circ (x+ B)^{-1}\circ x^4\nonumber\\
		=&\,\sum_{i=0}^{m-1}\left(S_{m-2-i}^{2^{i+1}}+D^{1-2^{i+1}}S_i\right)x^{2^i}\circ (x+B)\circ x^4\nonumber\\
		=&\,\sum_{i=0}^{m-1}\left(S_{m-2-i}^{2^{i+1}}+D^{1-2^{i+1}}S_i\right)(x^4+B)^{2^i},
	\end{align}
	where $S_i$ is a sequence with $S_{-1}=0,$ $S_0=1, $ $S_i=C^{2^{i-1}}S_{i-1}+D^{2^{i-1}}S_{i-2}.$
	Lemma \ref{th1} and \eqref{2eq1+(q
		^2+q)/4}  implies that the compositional inverse of $P(x)$ over $\gf_{q^2}$ is 
	$$P^{-1}(x)=x+b\left(\delta+\sum_{i=0}^{m-1}\left(S_{m-2-i}^{2^{i+1}}+D^{1-2^{i+1}}S_i\right)((x^q+x)^4+B)^{2^i}\right)^{1+(q^2+q)/4}.$$	
	We are done. 		
\end{proof}
Similarly, we have the following theorem.  We only give $\tau^4(x)=(x+b^4\delta^{q+1}\Tr_{q^2/q}(\delta))\circ (x^4+b^4\Tr_{q^2/q}(\delta)x^2+b^4\Tr_{q^2/q}(\delta)^2x)$ here. 

\begin{theorem}\label{th27}
	If $q\geq 4$ and the polynomial
	$$P(x)=b(x^q+x+\delta)^{q(2q+1)/4}+x$$ permutes $\gf_{q^2}$, then the compositional inverse of $P(x)$ over $\gf_{q^2}$ is 
	$$P^{-1}(x)=x+b\left(\delta+\sum_{i=0}^{m-1}\left(S_{m-2-i}^{2^{i+1}}+D^{1-2^{i+1}}S_i\right)\left(x^{4q}+x^4+B\right)^{2^i}\right)^{q(2q+1)/4},$$
	where	 $B=b^4\delta^{q+1}\Tr_{q^2/q}(\delta),$  $C=b^4\Tr_{q^2/q}(\delta),$ $D= b^4\Tr_{q^2/q}(\delta)^2,$ and $S_i$ is a sequence with $S_{-1}=0,$ $S_0=1, $ $S_i=C^{2^{i-1}}S_{i-1}+D^{2^{i-1}}S_{i-2}.$
\end{theorem}

Similarly, we have the following result. We only give $\tau(x)=b\Tr_{q^2/q}(\delta)x^4+ b\Tr_{q^2/q}(\delta)^3x^2
+x+  b\delta^{2q+2}\Tr_{q^2/q}(\delta)$ here. 
\begin{theorem}
	If the polynomial
	$P(x)=b(x^q+x+\delta)^{2q+3}+x$ permutes $\gf_{q^2}$, then the compositional inverse of $P(x)$ over $\gf_{q^2}$ is
	$$P^{-1}(x)=x+b\left(\delta+\sum_{i=0}^{m-1}\left(D^{2^i}S_{m-2-i}^{2^{i+1}}+D^{1-2^{i}}S_i\right)\left( x^q+x+B\right)^{2^i}\right)^{2q+3},$$
	where $A= b\Tr_{q^2/q}(\delta), B=b\delta^{2q+2}\Tr_{q^2/q}(\delta), C=\Tr_{q^2/q}(\delta)^2, D=1/A,$ and $S_i$ is a sequence with $S_{-1}=0,$ $S_0=1, $ $S_i=C^{2^{i-1}}S_{i-1}+D^{2^{i-1}}S_{i-2}.$

\end{theorem}

Similarly, we have the following result. We only give $\tau(x)=(b\Tr_{q^2/q}(\delta)^{2^i}x+b\Tr_{q^2/q}(\delta)^{2^i}\delta^{2q+2})\circ (x^4+\Tr_{q^2/q}(\delta)^2x^2+b^{-1}\Tr_{q^2/q}(\delta)^{-2^i}x)$ here.

\begin{theorem}\label{th25}
	For a non-negative integer $i$, if the polynomial
	$$P(x)=b(x^q+x+\delta)^{2q+2^i+2}+x$$ permutes $\gf_{q^2},$ then the compositional inverse of $P(x)$ over $\gf_{q^2}$ is 
	$$P^{-1}(x)=x+b\left(\delta+\sum_{i=0}^{m-1}\left(D^{2^i}S_{m-2-i}^{2^{i+1}}+D^{1-2^{i}}S_i\right)\left( x^q+x+B\right)^{2^i}\right)^{2q+2^i+2},$$
	where $A= b\Tr_{q^2/q}(\delta)^{2^i}, B=b\Tr_{q^2/q}(\delta)^{2^i}\delta^{2q+2}, C=\Tr_{q^2/q}(\delta)^2, D=1/A,$ and $S_i$ is a sequence with $S_{-1}=0,$ $S_0=1, $ $S_i=C^{2^{i-1}}S_{i-1}+D^{2^{i-1}}S_{i-2}.$	
\end{theorem}

Similarly, we have the following result. We only give $
\tau(x)=(b\Tr_{q^2/q}(\delta)^2x+b\delta^{6q}+b\delta^6)\circ	(x^4+\Tr_{q^2/q}(\delta)^2x^2+b^{-1}\Tr_{q^2/q}(\delta)^{-2}x).		
$

\begin{theorem}\label{th29}
	If the polynomial
	$$P(x)=b(x^q+x+\delta)^{6}+x$$ permutes $\gf_{q^2},$ then the compositional inverses of $P(x)$ over $\gf_{q^2}$ is 
	$$P^{-1}(x)=x+b\left(\delta+\sum_{i=0}^{m-1}\left(D^{2^i}S_{m-2-i}^{2^{i+1}}+D^{1-2^{i}}S_i\right)\left( x^q+x+B\right)^{2^i}\right)^{6},$$
	where $A=b\Tr_{q^2/q}(\delta)^2,$ $B=b\delta^{6q}+b\delta^6,$ $C=\Tr_{q^2/q}(\delta)^2,$ $D=A^{-1}$ and  $S_i$ is a sequence with $S_{-1}=0,$ $S_0=1, $ $S_i=C^{2^{i-1}}S_{i-1}+D^{2^{i-1}}S_{i-2}.$
\end{theorem}

\subsection{ The form of $\tau(x)$ is $Ax^{p^i}+Bx+C (p^i>2)$
 }

\begin{theorem}\label{thq1}
	Let $q=p^m$ with an odd prime $p$ and a positive integer $m.$  Assume that $\delta \in \gf_{q^2}$, $b_1, b_2, b_3 \in \gf_q,$ and  $\varepsilon \in \gf_{q^2}$ is a primitive element of $\gf_{q^2}.$ For a non-negative integer $i$ with $i< 2m$, let $j,d $ be a non-negative integers with $j<m$,  $j\equiv i \pmod{m},$ and $d=\gcd(m, j).$ For odd integers $d_1, d_2, d_3$ and $t=(q+1)/2,$
let $$P(x)=b_1\varepsilon^{td_1}(x^q+x+\delta)^{p^i+q}+b_2\varepsilon^{td_2}(x^q+x+\delta)^{p^i+1}+b_3\varepsilon^{td_3}(x^q+x+\delta)^{2p^i}-x$$ be a  polynomial over $\gf_{q^2}.$ \\	
	If $i=0, m$, then $P(x)$ is a permutation polynomial over $\gf_{q^2}$ if and only if $A-B\neq0.$ Moreover, if the latter condition holds,  the compositional inverse of $P(x)$ over $\gf_{q^2}$ is 
\begin{align*}P^{-1}(x)=&\,b_1\varepsilon^{td_1}((A-B)^{-1}(x^q+x)-(A-B)^{-1}C+\delta)^{p^i+q}\\
	&\,+b_2\varepsilon^{td_2}((A-B)^{-1}(x^q+x)-(A-B)^{-1}C+\delta)^{p^i+1}\\
	&\,+b_3\varepsilon^{td_3}((A-B)^{-1}(x^q+x)-(A-B)^{-1}C+\delta)^{2p^i}-x.
\end{align*}
If $i \neq 0, m,$  $A=0, $ and $B\neq0,$  then $P(x)$ is a permutation  polynomial over $\gf_{q^2}$ and  the compositional inverse of $P(x)$ over $\gf_{q^2}$ is 
\begin{align*}P^{-1}(x)=&\,b_1\varepsilon^{td_1}(-B^{-1}(x^q+x)+B^{-1}C+\delta)^{p^i+q}\\
	&\,+b_2\varepsilon^{td_2}(-B^{-1}(x^q+x)+B^{-1}C+\delta)^{p^i+1}\\
	&\,+b_3\varepsilon^{td_3}(-B^{-1}(x^q+x)+B^{-1}C+\delta)^{2p^i}-x.
\end{align*}
If $i \neq 0, m,$  $A\neq0, $ and $B=0,$ then $P(x)$ is a permutation  polynomial over $\gf_{q^2}$ and   the compositional inverse of $P(x)$ over $\gf_{q^2}$ is 
\begin{align*}
	P^{-1}(x)=&\, b_1\varepsilon^{td_1}(A^{-p^{m-j}}(x^q+x)^{p^{m-j}}-(C/A)^{p^{m-j}}+\delta)^{p^i+q}\\
	&\,+b_2\varepsilon^{td_2}(A^{-p^{m-j}}(x^q+x)^{p^{m-j}}-(C/A)^{p^{m-j}}+\delta)^{p^i+1}\\
	&\,+b_3\varepsilon^{td_3}(A^{-p^{m-j}}(x^q+x)^{p^{m-j}}-(C/A)^{p^{m-j}}+\delta)^{2p^i}-x.
\end{align*}
If  $i \neq 0, m$ and  $AB\neq0, $ then $P(x)$ permutes $\gf_{q^2}$ if and only if $N_{p^m/p^d}(B/A)\neq1$. Moreover, if  $P(x)$ permutes $\gf_{q^2}$,  then the compositional inverse of $P(x)$ over $\gf_{q^2}$ is 	\begin{align*}
	P^{-1}(x)=&\, b_1\varepsilon^{td_1}\left(\delta+\frac{N_{p^m/p^d}(\frac{B}{A})}{1-N_{p^m/p^d}\left(\frac{B}{A}\right)}\sum_{k=0}^{m/d-1}\left(\frac{A}{B}\right)^{\frac{p^{(k+1)j}-1}{p^j-1}}\left(\frac{x^q+x}{A}-\frac{C}{A}\right)^{p^{kj}}\right)^{p^i+q}\\	
	+&\,b_2\varepsilon^{td_2}\left(\delta+\frac{N_{p^m/p^d}(\frac{B}{A})}{1-N_{p^m/p^d}\left(\frac{B}{A}\right)}\sum_{k=0}^{m/d-1}\left(\frac{A}{B}\right)^{\frac{p^{(k+1)j}-1}{p^j-1}}\left(\frac{x^q+x}{A}-\frac{C}{A}\right)^{p^{kj}}\right)^{p^i+1}\\
	+&\,b_3\varepsilon^{td_3}\left(\delta+\frac{N_{p^m/p^d}(\frac{B}{A})}{1-N_{p^m/p^d}\left(\frac{B}{A}\right)}\sum_{k=0}^{m/d-1}\left(\frac{A}{B}\right)^{\frac{p^{(k+1)j}-1}{p^j-1}}\left(\frac{x^q+x}{A}-\frac{C}{A}\right)^{p^{kj}}\right)^{2p^i}\\-&\,x.
\end{align*}
Here, \begin{align*} A=&\,(\delta^q-\delta)(b_1\varepsilon^{td_1}-b_2\varepsilon^{td_2})+2b_3\varepsilon^{td_3}(\delta^{p^i}-\delta^{p^iq}),\\ B=&\,1-(\delta^{p^iq}-\delta^{p^i})(b_1\varepsilon^{td_1}+b_2\varepsilon^{td_2}),\\  C=&\,(\delta^{p^i+q}-\delta^{p^iq+1})b_1\varepsilon^{td_1}+(\delta^{p^i+1}-\delta^{p^iq+q})b_2\varepsilon^{td_2}+(\delta^{2p^i}-\delta^{2p^iq})b_3\varepsilon^{td_3}.
\end{align*}
\end{theorem}

\begin{proof}
	Since $\varepsilon \in \gf_{q^2}$ is a primitive element of $\gf_{q^2}$ and $t=(q+1)/2$, we have $\varepsilon^{(q-1)t}=-1,$ or 
	\begin{equation}\label{a1}
		\varepsilon^{qt}=-\varepsilon^t
	\end{equation}
	Lemma \ref{th1} yields that $P(x)$ permutes $\gf_{q^2}$
	if and only if $\tau(x)=b_1\varepsilon^{qtd_1}(x+\delta)^{p^iq+1}+b_1\varepsilon^{td_1}(x+\delta)^{p^i+q}+b_2\varepsilon^{qtd_2}(x+\delta)^{p^iq+q}+b_2\varepsilon^{td_2}(x+\delta)^{p^i+1}+b_3\varepsilon^{qtd_3}(x+\delta)^{2p^iq}+b_3\varepsilon^{td_3}(x+\delta)^{2p^i}-x.$ Note that $\tau(x)$ is a polynomial over $\gf_q.$
	Moreover, by \eqref{a1},  we have
	{\small	\begin{align}\label{1eqq+12q}
			\tau(x)=&\,b_1\varepsilon^{qtd_1}(x+\delta)^{p^iq+1}+b_1\varepsilon^{td_1}(x+\delta)^{p^i+q}+b_2\varepsilon^{qtd_2}(x+\delta)^{p^iq+q}\nonumber\\
			&\,+	b_2\varepsilon^{td_2}(x+\delta)^{p^i+1}+b_3\varepsilon^{qtd_3}(x+\delta)^{2p^iq}+b_3\varepsilon^{td_3}(x+\delta)^{2p^i}-x\nonumber\\
			=&\, -b_1\varepsilon^{td_1}(x^{p^i+1}+\delta x^{p^i}+\delta^{p^iq}x+\delta^{p^iq+1})+b_1\varepsilon^{td_1}(x^{p^i+1}+\delta^q x^{p^i}+\delta^{p^i}x+\delta^{p^i+q})\nonumber\\
		&\,-b_2\varepsilon^{td_2}(x^{p^i+1}+\delta^q x^{p^i}+\delta^{p^iq}x+\delta^{p^iq+q})+b_2\varepsilon^{td_2}(x^{p^i+1}+\delta x^{p^i}+\delta^{p^i}x+\delta^{p^i+1})\nonumber\\
		&\,	-b_3\varepsilon^{td_3}(x^{2p^i}+2\delta^{p^iq} x^{p^i}+\delta^{2p^iq})+b_3\varepsilon^{td_3}(x^{2p^i}+2\delta^{p^i} x^{p^i}+\delta^{2p^i})-x\nonumber\\
			=&\, Ax^{p^j}-Bx+C \end{align}}
	where $j$ is a non-negative integer with $j<m$ and $j\equiv i \pmod{m}$ and
	\begin{align*} A=&\,(\delta^q-\delta)(b_1\varepsilon^{td_1}-b_2\varepsilon^{td_2})+2b_3\varepsilon^{td_3}(\delta^{p^i}-\delta^{p^iq}),\\ B=&\,1-(\delta^{p^iq}-\delta^{p^i})(b_1\varepsilon^{td_1}+b_2\varepsilon^{td_2}),\\  C=&\,(\delta^{p^i+q}-\delta^{p^iq+1})b_1\varepsilon^{td_1}+(\delta^{p^i+1}-\delta^{p^iq+q})b_2\varepsilon^{td_2}+(\delta^{2p^i}-\delta^{2p^iq}-)b_3\varepsilon^{td_3}.
	\end{align*}
	
	If $i=0, m$, then $\tau(x)=(A-B)x+C$ by \eqref{1eqq+12q}, so $\tau(x)$ permutes $\gf_q$ if and only if $A-B\neq0.$ Moreover, if the latter condition holds, then  $\tau^{-1}(x)=(A-B)^{-1}x-(A-B)^{-1}C$ by \eqref{1eqq+12q} and $P(x)$ permutes $\gf_{q^2}$. Hence, the compositional inverse of $P(x)$ over $\gf_{q^2}$ is 
	\begin{align*}P^{-1}(x)=&\,b_1\varepsilon^{td_1}((A-B)^{-1}(x^q+x)-(A-B)^{-1}C+\delta)^{p^i+q}\\
		&\,+b_2\varepsilon^{td_2}((A-B)^{-1}(x^q+x)-(A-B)^{-1}C+\delta)^{p^i+1}\\
		&\,+b_3\varepsilon^{td_3}((A-B)^{-1}(x^q+x)-(A-B)^{-1}C+\delta)^{2p^i}-x.
		\end{align*}

	Next, we will do the case  $i\neq 0, m.$
	Notice that $\tau(x)$ is a affine polynomial over $\gf_q.$ Then $\tau(x)$ permutes $\gf_q$ if and only if $A=0$, or $B=0, $ or  $N_{p^m/p^d}(B/A)\neq1$ by Lemma \ref{binomial}, where $d=\gcd(m, j).$ 
	
	Moreover, if $A=0$ and $B\neq0$, then by \eqref{1eqq+12q}, we get  $\tau(x)=-Bx+C,$ and so $\tau^{-1}(x)=-B^{-1}x+B^{-1}C.$ This yields that  the compositional inverse of $P(x)$ over $\gf_{q^2}$ is 
			\begin{align*}P^{-1}(x)=&\,b_1\varepsilon^{td_1}(-B^{-1}(x^q+x)+B^{-1}C+\delta)^{p^i+q}\\
		&\,+b_2\varepsilon^{td_2}(-B^{-1}(x^q+x)+B^{-1}C+\delta)^{p^i+1}\\
		&\,+b_3\varepsilon^{td_3}(-B^{-1}(x^q+x)+B^{-1}C+\delta)^{2p^i}-x.
	\end{align*}
	by Lemma \ref{th1}.
	
	If $A\neq0$ and $B= 0$, then by \eqref{1eqq+12q}, we obtain $\tau(x)=Ax^{p^j}+C,$ and so $\tau^{-1}(x)=
	A^{-p^{m-j}}x^{p^{m-j}}-(C/A)^{p^{m-j}}.$
	Consequently, by Lemma \ref{th1}, the compositional inverse of $P(x)$ over $\gf_{q^2}$ is 
	\begin{align*}
		P^{-1}(x)=&\, b_1\varepsilon^{td_1}(A^{-p^{m-j}}(x^q+x)^{p^{m-j}}-(C/A)^{p^{m-j}}+\delta)^{p^i+q}\\
		&\,+b_2\varepsilon^{td_2}(A^{-p^{m-j}}(x^q+x)^{p^{m-j}}-(C/A)^{p^{m-j}}+\delta)^{p^i+1}\\
		&\,+b_3\varepsilon^{td_3}(A^{-p^{m-j}}(x^q+x)^{p^{m-j}}-(C/A)^{p^{m-j}}+\delta)^{2p^i}-x.
	\end{align*}
	
	If  $AB\neq0$ and $N_{p^m/p^d}(B/A)\neq1$,  then $\tau(x)=Ax^{p^j}-Bx+C=(Ax+C)\circ (x^{p^j}-Bx/A)$ permutes $\gf_q$, and so,  by Lemma \ref{binomial}, we have 
	\begin{align*}
		\tau^{-1}(x)=&\,\left(\frac{N_{p^m/p^d}(B/A)}{1-N_{p^m/p^d}(B/A)}\sum_{k=0}^{m/d-1}(B/A)^{-\frac{p^{(k+1)j}-1}{p^j-1}}x^{p^{kj}}\right)\circ (A^{-1}x-A^{-1}C)\\
		=&\,\frac{N_{p^m/p^d}(B/A)}{1-N_{p^m/p^d}(B/A)}\sum_{k=0}^{m/d-1}(A/B)^{\frac{p^{(k+1)j}-1}{p^j-1}}(A^{-1}x-A^{-1}C)^{p^{kj}}.
	\end{align*}
	
	Hence, it follows from Lemma \ref{th1} that  the compositional inverse of $P(x)$ over $\gf_{q^2}$ is 
	\begin{align*}
	P^{-1}(x)=&\, b_1\varepsilon^{td_1}\left(\delta+\frac{N_{p^m/p^d}(\frac{B}{A})}{1-N_{p^m/p^d}\left(\frac{B}{A}\right)}\sum_{k=0}^{m/d-1}\left(\frac{A}{B}\right)^{\frac{p^{(k+1)j}-1}{p^j-1}}\left(\frac{x^q+x}{A}-\frac{C}{A}\right)^{p^{kj}}\right)^{p^i+q}\\	
	+&\,b_2\varepsilon^{td_2}\left(\delta+\frac{N_{p^m/p^d}(\frac{B}{A})}{1-N_{p^m/p^d}\left(\frac{B}{A}\right)}\sum_{k=0}^{m/d-1}\left(\frac{A}{B}\right)^{\frac{p^{(k+1)j}-1}{p^j-1}}\left(\frac{x^q+x}{A}-\frac{C}{A}\right)^{p^{kj}}\right)^{p^i+1}\\
+&\,b_3\varepsilon^{td_3}\left(\delta+\frac{N_{p^m/p^d}(\frac{B}{A})}{1-N_{p^m/p^d}\left(\frac{B}{A}\right)}\sum_{k=0}^{m/d-1}\left(\frac{A}{B}\right)^{\frac{p^{(k+1)j}-1}{p^j-1}}\left(\frac{x^q+x}{A}-\frac{C}{A}\right)^{p^{kj}}\right)^{2p^i}\\-&\,x.
	\end{align*}
	This completes the proof.
\end{proof}
Similarly, we have following result. We omit the details.
\begin{theorem}\label{thq1}
	Let $q=2^m$ with  a positive integer $m.$  Assume that $\delta \in \gf_{q^2}$, $b_1, b_2 \in \gf_q,$ and $b_3 \in \gf_{q^2}. $  For a non-negative integer $i$ with $i< 2m$, let  $j, d$ be a non-negative integers with $j<m, $  $j\equiv i \pmod{m},$ and $d=\gcd(m, j).$ 	
And	let $$P(x)=b_1(x^q+x+\delta)^{2^i+q}+b_2(x^q+x+\delta)^{2^i+1}+b_3(x^q+x+\delta)^{2^i}+x$$ be a  polynomial over $\gf_{q^2}.$ \\
	If $i=0, m$, then $P(x)$ is a permutation polynomial over $\gf_{q^2}$ if and only if $A+B\neq0.$ Moreover, if the latter condition holds,  the compositional inverse of $P(x)$ over $\gf_{q^2}$ is 
	\begin{align*}P^{-1}(x)=&\,b_1((A+B)^{-1}(x^q+x)+(A+B)^{-1}C+\delta)^{2^i+q}\\
		&\,+b_2((A+B)^{-1}(x^q+x)+(A+B)^{-1}C+\delta)^{p^i+1}\\
		&\,+b_3((A+B)^{-1}(x^q+x)+(A+B)^{-1}C+\delta)^{2^i}+x.
	\end{align*}
	If $i \neq 0, m,$  $A=0, $ and $B\neq0,$  then $P(x)$ is a permutation  polynomial over $\gf_{q^2}$ and  the compositional inverse of $P(x)$ over $\gf_{q^2}$ is 
	\begin{align*}P^{-1}(x)=&\,b_1(B^{-1}(x^q+x)+B^{-1}C+\delta)^{2^i+q}\\
		&\,+b_2(B^{-1}(x^q+x)+B^{-1}C+\delta)^{2^i+1}\\
		&\,+b_3(B^{-1}(x^q+x)+B^{-1}C+\delta)^{2^i}+x.
	\end{align*}	
	If $i \neq 0, m,$  $A\neq0, $ and $B=0,$ then $P(x)$ is a permutation  polynomial over $\gf_{q^2}$ and the compositional inverse of $P(x)$ over $\gf_{q^2}$ is 
	\begin{align*}
		P^{-1}(x)=&\, b_1(A^{-p^{m-j}}(x^q+x)^{p^{m-j}}+(C/A)^{p^{m-j}}+\delta)^{2^i+q}\\
		&\,+b_2(A^{-p^{m-j}}(x^q+x)^{p^{m-j}}+(C/A)^{p^{m-j}}+\delta)^{p^i+1}\\
		&\,+b_3(A^{-p^{m-j}}(x^q+x)^{p^{m-j}}+(C/A)^{p^{m-j}}+\delta)^{2^i}+x.
	\end{align*}	
	If  $i \neq 0, m$ and  $AB\neq0, $ then $P(x)$ permutes $\gf_{q^2}$ if and only if $N_{2^m/2^d}(B/A)\neq1$. Moreover, if  $P(x)$ permutes $\gf_{q^2}$,  then the compositional inverse of $P(x)$ over $\gf_{q^2}$ is 	\begin{align*}
		P^{-1}(x)=&\, b_1\left(\delta+\frac{N_{2^m/2^d}(\frac{B}{A})}{1+N_{2^m/2^d}\left(\frac{B}{A}\right)}\sum_{k=0}^{m/d-1}\left(\frac{A}{B}\right)^{\frac{p^{(k+1)j}-1}{p^j-1}}\left(\frac{x^q+x}{A}+\frac{C}{A}\right)^{2^{kj}}\right)^{2^i+q}\\	
		+&\,b_2\left(\delta+\frac{N_{2^m/2^d}(\frac{B}{A})}{1+N_{2^m/2^d}\left(\frac{B}{A}\right)}\sum_{k=0}^{m/d-1}\left(\frac{A}{B}\right)^{\frac{p^{(k+1)j}-1}{p^j-1}}\left(\frac{x^q+x}{A}+\frac{C}{A}\right)^{p^{kj}}\right)^{2^i+1}\\
		+&\,b_3\left(\delta+\frac{N_{2^m/2^d}(\frac{B}{A})}{1+N_{2^m/2^d}\left(\frac{B}{A}\right)}\sum_{k=0}^{m/d-1}\left(\frac{A}{B}\right)^{\frac{p^{(k+1)j}-1}{p^j-1}}\left(\frac{x^q+x}{A}+\frac{C}{A}\right)^{p^{kj}}\right)^{2^i}+x.
	\end{align*}
	Here, \begin{align*} A=&\,(\delta+\delta^q)(b_1+b_2)+b_3+b_3^q,\\ B=&\,1+(\delta^{2^iq}+\delta^{2^i})(b_1+b_2),\\  C=&\,b_1(\delta^{2^iq+1}+\delta^{2^i+q})+b_2(\delta^{2^iq+q}+\delta^{2^i+1})+b_3^q\delta^{2^iq}+b_3a\delta^{2^i}.
	\end{align*}
\end{theorem}

Similarly, we also have following result. We only give $
	\tau(x)= -bx^3+\left(b(\delta^q-\delta)^2-1\right)x+b\delta^{q+1}(\delta^q+\delta)$ here.

\begin{theorem}\label{3 q+2th}
	Let $q=3^m$ with a  positive integer $m.$   Assume that $b \in \gf_{q}^*$ and $ \delta \in \gf_{q^2}.$  Let $A=-b, B=b\delta^{q+1}(\delta^q+\delta), C=(\delta^q-\delta)^2-b^{-1}.$
	And let
	$$P(x)=b(x^q+x+\delta)^{q+2}-x$$ permute $\gf_{q^2}.$\\	
	If $(\delta^q-\delta)^2-b^{-1}=0,$ then the compositional inverse of $P(x)$ over $\gf_{q^2}$ is  	$$P^{-1}(x)=b\left(A^{-q/3}((x^q+x)-B)^{q/3}+\delta\right)^{q+2}-x.$$\\	
	If	$(\delta^q-\delta)^2-b^{-1}$ is not a square in $\gf_q,$ then the compositional inverse of $P(x)$ over $\gf_{q^2}$ is 
	$$P^{-1}(x)=b\left(\delta+\frac{N_{3^m/3}(C)}{1-N_{3^m/3}(C)}\sum_{i=0}^{m-1}A^{-3^i}C^{-\frac{3^{i+1}-1}{2}}(x^q+x-B)^{3^i}\right)^{q+2}-x.$$
	
\end{theorem}





\noindent
{\bf Acknowledgments}
\\

 P. Yuan was supported by the National Natural Science Foundation of China (Grant No. 12171163), Guangdong Basic and Applied Basic Research Foundation (Grant No. 2024A1515010589). D. Wu was supported by Guangdong Basic and Applied Basic Research Foundation (Grant No. 2020A1515111090).
\\

\end{document}